\documentclass[a4paper, 11pt]{amsart}
\usepackage{epic,latexsym,amssymb}
\usepackage{color}
\usepackage{tikz}
\usepackage{amsfonts,epsf,amsmath}

 \usepackage{amsmath, amsthm, amssymb}
\usepackage{amssymb}
\usepackage{graphicx}
\usepackage{algorithmic}
\usepackage{algorithm}

\newtheorem{thm}{Theorem}

\newtheorem{lem}[thm]{Lemma}
\newtheorem{cor}[thm]{Corollary}

\newtheorem{rem}{Remark}
\newtheorem{ex}{Example}

\newcommand{\leg}[2]{\left(\frac{#1}{#2}\right)}

\begin{document}
	
	\title{On ternary quadratic forms over the rational numbers}
	
	\author{Amir Jafari$^*$\\Farhood Rostamkhani}
	\date{}
	\address{Department of Mathematical Sciences, Sharif University of Technology, Tehran, Iran.}
	\email{amirjafa@gmail.com, farhood.rostamkhani@gmail.com}
	\maketitle
	
	\begin{abstract}
		In this note, we give an elementary proof of the following classical fact. Any positive definite ternary quadratic form over the rational numbers fails to represent infinitely many positive integers. For any ternary quadratic form (positive definite or indefinite), our method constructs certain congruence classes whose elements, up to a square factor, are the only elements not represented over the rational numbers by that form. In the case of a positive definite ternary form, we show that these classes are non-empty.
		This shows that the minimum number of variables in a positive definite quadratic form representing all positive integers is four. Our proof is very elementary and only uses quadratic reciprocity of Gauss. 
		
	\end{abstract}
	\maketitle
	
	\section{Introduction}
	It is a well-known theorem of Lagrange that any positive integer is a sum of at most four squares. In other words, the quadratic form $x^2+y^2+z^2+t^2$ represents any positive integer over the integers. However, the quadratic form $x^2+y^2+z^2$ can not represent any number of the form $8k+7$ even when we allow $x,y$, and $z$ to be rational numbers. One may wonder if $ax^2+by^2+cz^2$ for some carefully chosen positive coefficients, $a$, $b$, and $c$ will represent all positive integers that are large enough. Indeed, if we allow having negative coefficients, then a ternary quadratic form such as $x^2+y^2-z^2$ can represent any number since for an integer $n$, we have
	$$(n+1)^2+0^2-n^2=2n+1,$$
	$$n^2+1^2-(n-1)^2=2n.$$
	However, it is a classical result that no ternary form $ax^2+by^2+cz^2$ represents all large enough integers for positive coefficients. In this note, we prove the following theorem that implies this classical result in a very precise way.
	\begin{thm}\label{main}
		Let $a$, $b$, and $c$ be square-free integers that are pairwise relatively prime. The following statements hold for the ternary quadratic form $ax^2+by^2+cz^2$.
		
		\begin{enumerate}
			\item If $p$ is an odd prime factor of $a$ (resp. $b$ or $c$) and $\leg{-bc}{p}=-1$ (resp. $\leg{-ac}{p}=-1$ or $\leg{-ab}{p}=-1$) and $n$ is an integer such that $\leg{an/p}{p}=-1$ (resp. $\leg{bn/p}{p}=-1$ or $\leg{cn/p}{p}=-1$) then $np$ is not represented over the rationals. 
			
			\item If for all odd prime factors $p$ of $a$ (resp. $b$ and $c$), one has $\leg{-bc}{p}=1$ (resp. $\leg{-ac}{p}=1$ and $\leg{-ab}{p}=1$) and if $a$, $b$ and $c$ are {\emph{positive}}, then either $a\equiv b\equiv c\mod 4$ or one of $a$, $b$ or $c$ (say $a$) is even and $b+c\equiv a$ or $2a\mod 8$.
			
			\item If $a\equiv b\equiv c\mod 4$ then numbers congruent to $-abc\mod 8$ are not represented over the rationals and if one of $a$, $b$ or $c$ (say $a$) is even and $b+c\equiv a$ or $2a\mod 8$ then numbers congruent to $-abc\mod 16$ are not represented over the rationals.
		\end{enumerate}
	\end{thm}

	Here, $\leg{\cdot}{p}$ is used for the Legendre symbol.
	
	We will show, in section 4, that this theorem implies that similar conclusions are valid for any positive definite ternary quadratic form over the rational numbers without any extra assumptions. Also, we will show, with the aid of Hasse-Minkowski's theorem, that the numbers listed in part 1 and part 3 of the theorem above (up to a square factor) are the only numbers not represented by $ax^2+by^2+cz^2$ over the rationals. Classifying the numbers represented by a ternary quadratic form over the integers is a much more difficult problem that has not been done in general.
	
	As a corollary to the above theorem, we have the following result that was proved by Doyle and Williams \cite{DW} for the particular case when $x$, $y$, and $z$ are integers and with a larger modulus, but with no restrictive assumptions on the coefficients. These assumptions can be removed, as is done in section 4 of this note.
	
	\begin{cor}
		If $a$, $b$ and $c$ are square free, pairwise relatively prime positive integers then if $a$ and $b$ and $c$ are all odd (resp. one of $a$, $b$ or $c$ is even) then any number congruent to $-abc \mod 8(abc)^2$ (resp. $\mod (2abc)^2$) is not represented by $ax^2+by^2+cz^2$ over the rational numbers.
	\end{cor}
	\begin{proof} Let $p$ be an odd prime factor of $a$ such that $\leg{-bc}{p}=-1$ and $a=pa'$. Then, any number congruent to $-abc \mod p^2$ is of the form $np$ where $n$ is congruent to $-a'bc\mod p$ and hence $\leg{a'n}{p}=\leg{-bc}{p}=-1$ and by part 1 of the theorem above $np$ is not represented over the rationals. By symmetry one can replace $a$ with $b$ or $c$. So if we are in the part 2 case of the theorem, then thanks to part 3 of the theorem, we know that $-abc\mod 8$ or $-abc\mod 16$ are not represented over the rationals in the case when $a,b,c$ are all odd or when one of them is even. Note that when one of $a$, $b$ or $c$ is even then $(2abc)^2\equiv 0\mod 16$. This proves the corollary.
	\end{proof}

	The history of the problem studied in this article is very rich. It was Fermat who, in 1638, first stated without proof that any number is a sum of at most four squares. After some unsuccessful tries by Euler, Lagrange was able to prove it completely in 1770. Another conjecture of Fermat that any number is a sum of at most three triangular numbers (i.e., numbers of the form $n(n+1)/2$) was proved by Gauss in 1796 which he proved by showing that the only numbers not represented by $x^2+y^2+z^2$ over the integers are of the form $4^m(8k+7)$. Gauss published this result in his famous Disquisitiones Arithmeticae in 1801. In 1748, Euler conjectured that every odd integer is represented by the ternary quadratic form $x^2+y^2+2z^2$ over the integers. This conjecture was proved by Lebesgue in 1857 \cite{Le}. Dickson \cite{D1} in 1927 showed that the forms $x^2+2y^2+3z^2$ and $x^2+2y^2+4z^2$ also represent all odd integers over the integers. In 1995 Kaplansky \cite{K2} gave a list of 23 integral positive definite ternary quadratic forms that he claimed to be the only such forms (up to equivalence) that represent all odd integers over the integers. He showed the validity of his claim for 19 of these forms. Interestingly, the three forms found by Euler and Dickson were the only diagonal ternary quadratic forms in this list. Another interesting historical result is due to Ramanujan. In 1916, he wrote a paper \cite{R}, in which, among other things, he studied the ternary form $x^2+y^2+10z^2$. He showed that the only even numbers not represented by this form over the integers are of the form $4^m(16k+6)$, and observed the following list of odd numbers are not represented over the integers $$3, 7, 21, 31, 33, 43, 67, 79, 87, 133, 217, 219, 223, 253, 307, 391, \dots$$
	He mentioned that they do not seem to follow any simple law. It is not clear if Ramanujan believed this list was complete or not or if the set was finite or infinite. In 1927 Jones and Pall \cite{JP} showed that $679$ can also be added to this list. In 1941 Gupta \cite{G} found another number in this list, namely $2719$. Computer searches for numbers up to $2\times 10^{10}$ have not produced any new odd numbers. It was proved in 1990 by Duke and Schulze-Pillot in \cite{DS} that $x^2+y^2+10z^2$ represents all odd integers large enough, and hence this list is finite. Also, in 1997 Ono and Soundararajan \cite{OS} conjectured that Ramanujan's list with the two extra numbers $679$ and $2719$ is complete. They showed the validity of their conjecture, assuming the generalized Riemann hypothesis. 
	
	The study of numbers represented by ternary quadratic forms over the integers is a very active research theme. For example, besides odd numbers, one may be tempted to study other arithmetical progressions $dk+r$ for $k\in\{0,1,\dots\}$. A form representing all numbers in this progression over the integers is called $(d,r)$-universal. Sun in \cite{S1} and \cite{S2} proved that the forms $x^2+3y^2+24z^2$, $4x^2+3y^2+6z^2$ and $x^2+12y^2+6z^2$ are $(6,1)$-universal. In \cite{WS}, Wu and Sun were able to show that $2x^2+3y^2+10z^2$ is $(8,5)$-universal. Also, they showed that $x^2+3y^2+14z^2$ and $2x^2+3y^2+7z^2$ are $(14,7)$-universal.
	
	According to \cite{DW}, it is not known who was the first person to state that integral positive definite ternary quadratic forms can not represent all positive integers. It is stated in Albert \cite{A} in 1933 and also on page 142 of Conway's beautiful book \cite{Con}, where he gives a modern proof using $p$-adic equivalence of forms and isotropic forms.   Finally we mention two papers by Mordell \cite{M}, \cite{M2} of 1931 and 1932 that study the solvability of the equation $ax^2+by^2+cz^2+dt^2=0$ when $x,y,z$ and $t$ are integers not all equal to zero. As this is related to the representability of an integer by $ax^2+by^2+cz^2$ over the rationals, some of his results overlap ours.
	
	In this note, we give an elementary proof that uses only the quadratic reciprocity law of Gauss. Our methods are similar to the methods used in \cite{DW}. However, our proof is shorter and easier to follow and produces a more decisive result, i.e., non-representability over the rationals instead of the integers. Also, our list of numbers not represented over the rational numbers is complete.
	\\
	{\textbf{Acknowledgment.}} The authors wish to express their acknowledgment to a referee for providing very constructive remarks that improved the quality of this paper considerably.
	
	\section{Proof}
	
	In this section, we present a proof of Theorem \ref{main}. Throughout, the coefficients, $a$, $b$, and $c$ are non-zero, square-free, and pairwise relatively prime.
	\\
	\\
	{\textbf{Proof of the first part of Theorem 1.}}
	Assume $p|a$ and 
	in contrary, assume that for integers $x,y, z$ and $t$, we have 
	$$ax^2+by^2+cz^2=npt^2.$$
	We may assume that $p$ does not divide at least one of $x, y, z$ or $t$. If $y$ or $z$ are not divisible by $p$, then by taking the above equation modulo $p$, it follows that $-bc$ is a quadratic residue modulo $p$, contrary to the assumption of the lemma. So $y$ and $z$ are divisible by $p$. Simplifying the equation, it follows that 
	$$(a/p)x^2\equiv nt^2\pmod p.$$
	Since at least one of $x$ or $t$ is not divisible by $p$, $na/p$ is a quadratic residue modulo $p$, again contrary to the assumption of the lemma. Hence, $ax^2+by^2+cz^2$ does not represent $np$ over the rational numbers.
	\\
	\\
	{\textbf{Proof of the second part of Theorem 1.}}
	In this part $a$, $b$, and $c$ are assumed to be positive. The proof uses the quadratic reciprocity for the Jacobi symbol
	$$\leg{m}{n}=\prod_i \leg{m}{p_i}^{e_i}$$
	where $n=p_1^{e_1}\dots p_k^{e_k}$ is an odd number. We use the following well-known facts, see \cite{F}.
	$$\leg{-1}{n}=(-1)^{(n-1)/2},\quad \leg{2}{n}=(-1)^{(n^2-1)/8}\quad \mbox{$n$ is odd}$$
	$$\leg{n}{m}\leg{m}{n}=(-1)^{\frac{n-1}{2}\frac{m-1}{2}}\quad \mbox{$m$ and $n$ are odd and $(m,n)=1$}$$
	$$\leg{mm'}{n}=\leg{m}{n}\leg{m'}{n}\quad\mbox{$n$ is odd}$$
	The assumption implies when $a,b$ and $c$ are odd that $\leg{-bc}{a}=\leg{-ac}{b}=\leg{-ab}{c}=1$ and hence if we write these as 
	$$\leg{-1}{a}\leg{b}{a}\leg{c}{a}=1$$
	$$\leg{-1}{b}\leg{a}{b}\leg{c}{b}=1$$
	$$ \leg{-1}{c}\leg{a}{c}\leg{b}{c}=1$$ 
	and take the product of the three expressions and use Gauss reciprocity, we get
	$$(-1)^{\alpha+\beta+\gamma+\alpha\beta+\beta\gamma+\alpha\gamma}=1$$
	where $\alpha =(a-1)/2$, $\beta=(b-1)/2$ and $\gamma=(c-1)/2$.
	This implies that $\alpha,\beta$, and $\gamma$ have the same parity since if, for example, $\alpha$ is even and $\beta$ is odd, the exponent is congruent to $1$ modulo $2$. The first case of the lemma is proved.
	
	Now assume that $a=2a'$, where "a'" is an odd number. Then
	$$\leg{-1}{a'}\leg{b}{a'}\leg{c}{a'}=1$$
	$$\leg{-1}{b}\leg{2}{b}\leg{a'}{b}\leg{c}{b}=1$$
	$$ \leg{-1}{c}\leg{2}{c}\leg{a'}{c}\leg{b}{c}=1$$

	So if we let $\alpha=(a'-1)/2$, $\beta=(b-1)/2$, $\gamma=(c-1)/2$, $\beta'=(b^2-1)/8$ and $\gamma'=(c^2-1)/8$, it follows by multiplying this expressions and using the properties of the Jacobi symbol that 
	$$(-1)^{\alpha\beta+\alpha\gamma+\beta\gamma+\alpha+\beta+\gamma+\beta'+\gamma'}=1$$
	Equivalently
	$$8\alpha\beta+8\beta\gamma+8\beta\gamma+8\alpha+8\beta+8\beta'+8\gamma'=(a'+b+c)^2-(a')^2-8$$
	must be divisible by 16.
	This implies that $(b+c)(b+c+2a')$ is divisible by $8$ and not by $16$. Since $b$ and $c$ are odd hence $b+c\equiv 0, a,2a$ or $3a\pmod 8$. Since $(b+c)(b+c+2a')$ is divisible by $16$ when $b+c\equiv 0$ or $3a\pmod 8$ the second part of the lemma will be proved.
	\\
	\\
	{\textbf{Proof of the third part of Theorem 1.} The following is a stronger result.
		\begin{lem}
			The following statements hold for the congruence classes modulo $8$ or $16$ represented by $ax^2+by^2+cz^2$.
			\begin{enumerate}
				\item If $a$, $b$ and $c$ are odd integers then $ax^2+by^2+cz^2$ modulo $8$ represents all congruence classes except when $a\equiv b\equiv c \mod 4$ where the only congruence class that it does not represent is $-abc\mod 8$. And in this case, any number congruent to $-abc\mod 8$ is not represented over the rationals.
				\item If one of $a$, $b$ or $c$ (say $a$) is even then $ax^2+by^2+cz^2$ modulo $2^n$ represents all congruence classes except when $b+c\equiv a$ or $2a\mod 8$ where the only congruence class that it does not represent is $-abc\mod 16$. And in this case any number congruent to $-abc\mod 16$ is not represented over the rationals. 
			\end{enumerate}
		\end{lem}
		\begin{proof} 	If $a, b$ and $c$ are odd and $a\not\equiv b\mod 4$, then $ax^2+by^2\mod 4=\{0,1,3\}$. Also since $4c\equiv 4\mod 8$, if $ax^2+by^2\mod 8$ represents $n$, then $ax^2+by^2+cz^2\mod 8$ will represent $n+4$ as well, hence we find that $ax^2+by^2+cz^2\mod8$ contains $\{0,1,3,4,5,7\}$. Finally since $c$ is congruent to $a$ or $b$ (say $a$) $\mod 4$ hence $2a$ and $2a+4$ which are $2$ and $6\mod 8$ will be represented as well. 
			So, in the first case, except when $a\equiv b\equiv c \mod4$, all congruence classes modulo $8$ are represented. In this exceptional case, either they are all congruent modulo $8$, and we have $a(x^2+y^2+z^2)$ modulo $8$. Since $x^2+y^2+z^2$ modulo $8$ can be any congruence class except $-1$; hence we get all congruence classes except $-a$, which is the same as $-abc\mod 8$. This last fact is because $-abc\equiv-aa^2\equiv-a\mod 8$. If the coefficients are not congruent modulo $8$, we may assume without loss of generality that our form modulo $8$ is $a(x^2+y^2+5z^2)$. Since $x^2+y^2+5z^2$ modulo $8$ can be any congruence class except $-5$; hence we get all congruence classes except $-5a$, which is $-abc\mod 8$. To show that numbers $n$ congruent to $-abc \mod 8$ will not be represented over the rationals when $a\equiv b\equiv c \mod 4$, assume that $ax^2+by^2+cz^2=nt^2$ where $x,y,z,t$ have no common factor. If $t$ is odd then $ax^2+by^2+cz^2\equiv n\mod 8$ which is a contradiction. If $t$ is even, then two of $x$, $y$ or $z$ must be odd, and the third one must be even, say $x$, and $y$ are odd, and $z$ is even, then $a+b\equiv 0\mod 4$, which is also a contradiction.
			\\
			\\
			If $a=2a'$ with $a'$ being odd then $ax^2+by^2\mod4$ will represent all congruence classes. Since $4c\equiv 4\mod 8$, as before $ax^2+by^2+cz^2\mod8$ will represent all classes. Observe that $9b\equiv b+8\mod 16$ and $9c\equiv c+8\mod 16$. Hence if $ax^2+by^2+cz^2\mod 16$ represents $n$ and one of $y$ or $z$ (say $y$) is odd then by replacing $y$ with $3y$, $n+8$ will be represented as well. Since all integers modulo $8$ are represented, this implies that all odd integers modulo $16$ are represented.  Note that since $a'$ and $3a'$ are represented modulo $16$, hence by replacing $x,y$ and $z$ with $2x,2y$ and $2z$, we can represent $4a'=2a$ and $12a'=6a$ modulo $16$. It is also clear that $a$ and $4a$ are represented modulo $16$. So the only even numbers that remain modulo $16$ are $3a,5a,7a$. Observe that the only even integers modulo $16$ represented by $by^2+cz^2$ are $0,4b,4c, 4(b+c),b+c$ and $b+c+8\equiv b+c+4a\mod 16$.
			\\
			{\textbf{Case 1:}} If $b+c\equiv 0\mod 8$, then $4b$ and $4c$ will give $2a$ and $6a\mod 16$ and either $b+c$ or $b+c+4a$ give $4a\mod 16$. Adding $a$ from $ax^2$ term, we can represent $3a,5a,7a$ modulo $16$ as demanded.
			\\
			{\textbf{Case 2:}} If $b+c\equiv 3a\mod 8$, then we can represent $b+c$ and $b+c+4a$ these are $3a$ and $7a\mod 16$. Also, we can represent $4(b+c)\equiv 4a\mod 16$ and adding $a$, will represent $5a$ as well.
			\\
			{\textbf{Case 3:}} If $b+c\equiv 2a\mod 8$, the numbers of the form $4k+2$ that will be represented modulo $16$ are $a+b+c, a+b+c+4a, a+4b,a+4c$ and $a+4(b+c)$ which are $3a, 7a, a$ and hence $5a$ will not be represented. 
			To show that $5a\equiv -abc\mod 16$, it is enough to show that $bc\equiv 3\mod 8$. But since $b$ and $c$ are odd and $b+c$ is of the form $8k+4$ and hence either one of them must be $\equiv 1\mod 8$ and the other one $\equiv 3\mod 8$ or one of them is $\equiv -1\mod 8$ and the other one $\equiv -3\mod 8$, so $bc\equiv 3\mod 8$ as demanded.
			\\
			{\textbf{Case 4:}} If $b+c\equiv a\mod 8$, then we can represent $a+4b+4c\equiv 5a\mod 16$. We have now two cases, if $b\equiv c\equiv a'\mod4$ then we can represent $a+4a'=3a\mod 16$. The congruence classes of the form $4k+2\mod 16$, represented by the form are $b+c, b+c+4a, 4b+a,4c+a, 4a+b+c,4a+b+c+4a$ and none of them is $7a$. To show that $7a\equiv -abc\mod 16$, it is enough to show $bc\equiv 1\mod 8$. We must have either $b\equiv c\equiv a'\mod 8$ or $b\equiv c\equiv 5a'\mod 8$ and in both cases we have $bc\equiv 1\mod 8$.
			\\
			Similarly if $b\equiv c\equiv 3a'\mod 4$, then we can represent $a+4b\equiv 7a\mod 16$. The congruence classes of the form $4k+2\mod 16$, represented by the form are $b+c, b+c+4a, 4b+a,4c+a, 4a+b+c,4a+b+c+4a$ and none of them is $3a$.
			To show that $3a\equiv -abc\mod 16$, it is enough to show that $bc\equiv 5\mod 8$. But one of the two coefficients $b$ or $c$ must be congruent to $3a'$ and the other one congruent to $7a'$ modulo $8$ and hence $bc\equiv 5\mod 8$.
			\\
			\\
			Finally in the second case, if $b+c\equiv a$ or $2a\mod 8$, we need to show that any number $n\equiv -abc\mod 16$ is not represented over the rational numbers. Assume in the contrary that there exists $x,y,z$ and $t$, not all even and $t\ne 0$ such that 
			$ax^2+by^2+cz^2=nt^2$. If $t$ is even then since $n$ is even, $ax^2+by^2+cz^2\equiv 0\mod 8$. Since parity of $y$ and $z$ are the same, we are in the following cases.
			\\
			If $x, y$ and $z$ are odd then 
			$$ax^2+by^2+cz^2\equiv a+b+c\equiv 2a\quad\mbox{or}\quad 3a \not\equiv 0\mod 8$$
			If $x$ is even and $y$ and $z$ are odd then 
			$$ax^2+by^2+cz^2\equiv b+c\equiv a\quad\mbox{or}\quad 2a\not\equiv 0\mod 8$$
			If $x$ is odd and $y$ and $z$ are even then
			$$ax^2+by^2+cz^2\equiv a\not\equiv 0\mod 4$$
			Hence it follows that $x,y$, and $z$ must be even, which is a contradiction. This implies that $t$ is odd and hence 
			$$ax^2+by^2+cz^2=nt^2\equiv n\mod 16$$
			which we showed can not happen. This proves the claim.
		\end{proof}
		\section{An application of Hasse-Minkowski's Theorem}
		This section uses the so-called Hasse-Minkowski's theorem to determine exactly what numbers are represented by $ax^2+by^2+cz^2$ over the rational numbers. As mentioned in the introduction, a similar problem of determining all the numbers represented by that form over the integers is much more difficult and has not been done in general. 
		\\
		Hasse-Minkowski's theorem, in this case, asserts to be able to represent $n$ by $ax^2+by^2+cz^2$ over the rational numbers, we need to be able to represent it over the real numbers and modulo any prime power.
		\\
		We keep our assumption on the coefficients $a$, $b$, and $c$ as before. Namely, they are non-zero, square-free, and pairwise relatively prime. No assumption on their signs is made.
		\\
		To achieve our goal, we give several lemmas that are going to be used later.
		\begin{lem}
			If $p$ is an odd prime number that does not divide $ab$, then any integers $n$ prime to $p$ is represented by $ax^2+by^2+cz^2\mod p^m$ for any $m\ge 1$. If $p$ does not divide $c$, then all numbers are represented modulo $p^m$.
		\end{lem}
		\begin{proof} We prove this by induction on $m=1$. Since the classes of $ax^2$ and $n-by^2$ modulo $p$ is each of size $(p+1)/2$ so they intersect and hence one can find $x$ and $y$ not both divisible by $p$ (since $n$ is prime to $p$) such that $ax^2+by^2\equiv n\mod p$. Suppose that we have $x$ and $y$ not both multiples of $p$, say $x$ is not a multiple of $p$, such that $ax^2+by^2=n+kp^m$ for $m\ge 2$ and some integer $k$. Let $x'=x+k_1p^m$. Then
			$$a(x')^2+by^2=n+kp^m+2axk_1p^m+ak_1^2p^{2m}$$
			since $2m\ge m+1$, we need to take $k_1=-(2ax)^{-1}k\mod p$ for the induction to work. So we showed even $ax^2+by^2\mod p^m$ represents all prime to $p$ integers. If $p$ does not divide $c$, an $n$ is a multiple of $p$ then $n-c$ is prime to $p$ and is therefore represented by $ax^2+by^2\mod p^m$. hence if we let $z=1$ we have $n\equiv ax^2+by^2+c\mod p^m$. 
		\end{proof}
		
		\begin{lem} If $p$ is an odd prime number that divides $c$ then for any integer $n$ prime to $p$, $np$ is represented by $ax^2+by^2+cz^2\mod p^m$ for any $m\ge 1$, unless $\leg{-ab}{p}=\leg{nc/p}{p}=-1$ where it is not represented $\mod p^2$.
		\end{lem}
		\begin{proof}
			Note that if $\leg{-ab}{p}=1$ then $ax^2+by^2\equiv np\equiv 0\mod p$ has a solution that not both $x$ and $y$ are divisible by $p$ and the above proof works again. If $\leg{nc/p}{p}=1$, let $c=pc'$, then we can find $z$ such that $c'z^2\equiv n\mod p$ or $cz^2\equiv np\mod p^2$. 
			Suppose that $cz^2=np+kp^m$ for some $m\ge 2$ and $k$. Then since $n$ is prime to $p$, $z$ is also prime to $p$. Let $z'=z+k_1p^{m-1}$, then
			$$c(z')^2=np+kp^m+2c'zk_1p^{m}+k_1^2c'p^{2m-1}.$$
			Since $2m-1\ge m+1$, it is enough to take $k_1\equiv -(2c'z)^{-1}k\mod p$ for the induction to work. So, even $cz^2$ alone will represent $np$ modulo $p^m$. The second part of the lemma was done when we proved Theorem 1, part 1.
		\end{proof}
		\begin{lem}
			If $a$, $b$ and $c$ are odd integers and an integer $n$ is represented by $ax^2+by^2+cz^2\mod 8$ then it is represented $\mod 2^m$ for any $m\ge 3$.
		\end{lem}
		\begin{proof}
			It is enough to show this for the case where $n$ is not divisible by $4$, since we can replace $x$, $y$ and $z$ by $2^kx$, $2^ky$ and $2^kz$ for some $k$. We prove the claim by induction on $m$. Suppose that we can find $x.y,z$ such that $ax^2+by^2+cz^2=n+k2^m$ for $m\ge 3$. Since $n$ is not a multiple of $4$ hence at least one of $x,y$ or $z$ (say $x$) is odd. Define $x'=x+k_12^{m-1}$ then 
			$$a(x')^2+by^2+cz^2=n+k2^m+2^mak_1x+ ak_1^22^{2m-2}$$
			since $2m-2\ge m+1$, we only need to take $k_1=k$ to make the induction work.
		\end{proof}
		\begin{lem}
			If one of $a$, $b$ or $c$ is even and and an integer $n$ is represented by $ax^2+by^2+cz^2\mod 16$ then it is represented $\mod 2^m$ for any $m\ge 4$.
		\end{lem}
		\begin{proof}
			As the previous part we may assume $n$ is not divisible by $4$. Assume $ax^2+by^2+cz^2=n+k2^m$ for some $m\ge 4$. Since $n$ is not a multiple of $4$ one of $x$, $y$ or $z$ is odd. Assume $x$ is odd. If $a$ is odd as well the same proof as before makes the induction work. Assume $a=2a'$ where $a'$ is an odd integer. Let $x'=x+k_12^{m-2}$ then 
			$$a(x')^2+by^2+cz^2=n+k2^m+2^ma'k_1x+a'k_1^22^{2m-3}.$$
			Since $2m-3\ge m+1$, we only need to take $k_1=k$ to make the induction work.
		\end{proof}
		\begin{thm}
			Let $N$ be a non-zero rational number and let $N=n m^2$ where $n$ is a square-free integer and $m$ is a rational number. Then $N$ is represented by $ax^2+by^2+cz^2$ over the rational numbers if and only if the following conditions hold.
			\begin{enumerate}
				\item If $a$, $b$ and $c$ are positive (resp. negative) then $N$ is positive (resp. negative).
				\item If $a$, $b$ and $c$ are odd and $a\equiv b\equiv c\mod 4$ then $n\not\equiv -abc\mod 8$.
				\item If one of $a$, $b$ or $c$ (say $a$) is even and $b+c\equiv a$ or $2a\mod 8$ then $n\not\equiv -abc\mod 16$.
				\item For all odd primes $p|n$ and $p|a$ either $\leg{-bc}{p}=1$ or $\leg{na/p^2}{p}=1$. Similar results should hold symmetrically for $b$ and $c$.
			\end{enumerate}
		\end{thm}
		\begin{proof} Note that $N$ is represented over the rationals by $ax^2+by^2+cz^2$ if an only if $n$ is represented by the rationals. It is enough to replace $x,y$ and $z$ with $mx,my$ and $mz$. Now according to Theorem 1 the conditions above are necessary. Also, Lemmas 3, 6 and 7 imply that $n$ is represented by $ax^2+by^2+cz^2\mod 2^m$ for any $m$. And for odd prime $p$, Lemmas 4 and 5 show that it is also represented by $ax^2+by^2+cz^2\mod p^m$. Condition 1 in the theorem, implies that $n$ is represented over the real numbers and now Hasse-Minkowski's theorem finishes the proof.
			\end{proof}
		\begin{ex} For instance, consider the Ramanujan ternary form $x^2+y^2+10z^2$. Since $\leg{-1}{5}=1$ and $a+b\equiv c\mod 8$, i.e. $2\equiv 10\mod 8$, the only positive numbers up to a square factor not represented over the rational numbers are those $\equiv -10\mod 16$. That is numbers of the form $16k+6$. Also, since for an odd integer $x$, $x^2\equiv 1$ or $9\mod 16$, hence the only positive integers not represented over the rationals are $4^m(16k+6)$. This is different from representing numbers over integers. It was mentioned in the introduction that without assuming the generalized Riemann hypothesis, we do not know all the odd numbers not represented by this form over the integers. All odd positive integers are represented over the rational numbers by this form. For example $3=(\frac{1}{2})^2+(\frac{1}{2})^2+10(\frac{1}{2})^2$, although $3$ is not represented over the integers.
		\end{ex}
		
		An interesting consequence of the theorem above is the following result that is in some sense converse to the corollary 1. 
		\begin{thm} If $a, b$ and $c$ are non-zero pairwise relatively prime integers that are square-free and are not of the same sign then $ax^2+by^2+cz^2$ represents all integers over the integers if and only if the following congruence is solvable.
		$$ax^2+by^2+cz^2\equiv -abc\mod (abc)^2.$$
		\end{thm}
		\begin{proof}
		Suppose that the congruence $ax^2+by^2+cz^2\equiv -abc\mod (abc)^2$ is solvable.
		If for a prime $p|c$, we have $\leg{-ab}{p}=-1$ then according to lemma 5, $-abc$ is not represented by $ax^2+by^2+cz^2\mod p^2$, because $\leg{(-abc/p)(c/p)}{p}=\leg{-ab}{p}=-1$, which contradicts our assumption. This implies that for any prime $p|c$, we have $\leg{-ab}{p}=1$. By the Chinese remainder theorem, this implies that $-ab$ is a quadratic residue modulo $c$. Similarly $-bc$ and $-ac$ are quadratic residues modulo respectively $a$ and $b$. Now, we will se a well-known theorem proved by Legendre in 1785 (see \cite{L}, pp 509-513) that states, with these conditions we have a solution 
		$$ax_0^2+by_0^2+cz_0^2=0$$
		in integers where not all three numbers $x_0,y_0$ and $z_0$ are zero. We may assume that the greatest common divisor of $x_0,y_0,z_0$ is one. This implies that these numbers are pairwise relatively prime, since we have assumed that the coefficients $a,b$ and $c$ are square free and if a prime factor divides $x_0$ and $y_0$ (for example) then it must divide $z_0$ as well. From this fact, it follows that the greatest common divisor of $ax_0$, $by_0$ and $cz_0$ is one. Therefore there are integers $x_1$, $y_1$ and $z_1$ such that 
		$$ax_0x_1+by_0y_1+cz_0z_1=1.$$
		Now it follows that
		\begin{equation}
		a(kx_0+x_1)^2+b(ky_0+y_1)^2+c(kz_0+z_1)^2=2k+(ax_1^2+by_1^2+cz_1^2).
		\end{equation}
		Assume first that $a,b$ and $c$ are odd. Then at least one of $x_0,y_0$ or $z_0$ (say $x_0$) is even the other two are odd. Note that $x_2=x_1+by_0$, $y_2=y_1-ax_0$ and $z_2=z_1$ satisfies
		$$ax_0x_2+by_0y_2+cz_0z_2=1$$
		Hence
			\begin{equation}
		a(kx_0+x_2)^2+b(ky_0+y_2)^2+c(kz_0+z_2)^2=2k+(ax_2^2+by_2^2+cz_2^2).
		\end{equation}
		Since $ax_1^2+by_1^2+cz_1^2$ and $ax_2^2+by_2^2+cz_2^2$ have different parities, the equations (1) and (2) above will represent all integers.
	Now assume that $a$ is even and $b$ and $c$ are odd. Then since $y_0$ and $z_0$ are coprime, they need to be odd. Hence $y_1$ and $z_1$ must be of different parities and this implies that $ax_1^2+by_1^2+cz_1^2$ is odd. So the equation (1) represents all odd integers.
	To represent all even integers, it is enough to represent numbers of the form $4k+2$. Since other numbers are of the form $4^m n$ where $n$ is either odd or $\equiv 2\mod 4$, which has been represented.  Now choose $x_1,y_1,z_1$ such that 
	$$ax_0x_1+by_0y_1+cz_0z_1=2$$
	and hence
	\begin{equation}
		a(kx_0+x_1)^2+b(ky_0+y_1)^2+c(kz_0+z_1)^2=4k+(ax_1^2+by_1^2+cz_1^2).
		\end{equation}	
		As before, we may replace $(x_1,y_1,z_1)$ with $(x_1+by_0,y_1-ax_0,z_1)$ if needed to assume $x_1$ is odd. Since $y_1$ and $z_1$ are of the same parity, if we replace $(x_1,y_1,z_1)$ with $(x_1, y_1+cz_0, z_1-by_0)$ we may assume that $y_1$ and $z_1$ are even. Hence
		$$ax_1^2+by_1^2+cz_1^2\equiv a\equiv 2\mod 4$$
		and equation (3) will represent all numbers $\equiv 2\mod 4$.
				\end{proof}
		
		\begin{rem} Inspecting the proof given above, one sees that the ternary quadratic form $ax^2+by^2+cz^2$ with $abc$ square-free represents all integers over the integers if and only if it represents zero non-trivially. For example $x^2+y^2-cz^2$ represents all integers if and only if $n$ is a sum of two coprime integers. 
		\end{rem}
		\section{Final Remarks}
		In this section, we remove extra assumptions that we have imposed on the coefficients $a$, $b$ and $c$ of the form $f=ax^2+by^2+cz^2$. To ease our exposition, denote $R(a,b,c)$ to be the set of rational numbers represented by $f$ over the rational numbers. Assume $a$, $b$ and $c$ be non-zero rational numbers. Let $a=r^2a_1, b=s^2b_1$ and $c=t^2c_1$, where $a_1,b_1$ and $c_1$ are square-free integers and $r$, $s$ and $t$ are rational numbers. Then by a change of variables $x$, $y$ and $z$ to $rx$, $sy$ and $tz$ it is clear that
		$R(a,b,c)=R(a_1,b_1,c_1).$
		Next let $d$ be the greatest common divisor of $a_1,b_1$ and $c_1$ and write $a_1=da_2$, $b_1=db_2$ and $c_1=dc_2$. We have
		$R(a_1,b_1,c_1)=d\cdot R(a_2,b_2,c_2).$
		To proceed, write $a_2=d_1d_2 a_3$, $b_2=d_2d_3b_3$ and $c_2=d_1d_3c_3$, where $d_1$, $d_2$ and $d_3$ are the greatest common divisor of respectively, $a_2$ and $b_2$; $a_2$ and $c_2$; and finally $b_2$ and $c_2$. We can deduce by factoring $d_1d_2d_3$ and then multiplying each coefficient with an appropriate square that
		$$R(a_2,b_2,c_2)=d_1d_2d_3R(a_3/d_3, b_3/d_1, c_3/d_2)=d_1d_2d_3R(a_3d_3, b_3d_1, c_3d_2).$$
		Now if we let $a'=a_3d_3, b'=b_3d_1$ and $c'=c_3d_2$, these numbers are square-free and pairwise relatively prime. We have
		\begin{equation}R(a,b,c)=dd_1d_2d_3R(a',b',c')\end{equation}
		We then get the following corollary that is a generalization of Theorem 2 of \cite{DW}.
		\begin{cor}
			Let $a,b$ and $c$ are non-zero positive rational numbers and $abc=R^2S$ where $S$ is square-free. Then with the notation above, every number $\equiv -S\mod 8d_1d_2d_3S^2$ will not be represented over the rational numbers. In particular, if $a, b$ and $c$ are positive integers then any number $\equiv -S\mod 8abcS$ and more particularly any number $\equiv-abc\mod 8(abc)^2$ can not be represented over the rational numbers.
		\end{cor}
		\begin{proof}
			From corollary 1, we know $a'x^2+b'y^2+c'z^2$ fails to represent all numbers $\equiv -a'b'c'\mod 8(a'b'c')^2$, so the original form fails to represent all numbers congruent to $-dd_1d_2d_3a'b'c'\mod 8dd_1d_2d_3(a'b'c')^2$. Now notice that $S=da_3b_3c_3$, so we can rewrite this as all numbers $\equiv -(d_1d_2d_3)^2S\mod 8 d(d_1d_2d_3)^3 (a_3b_3c_3)^2$. We can divide by a square, here $(d_1d_2d_3)^2$, without changing the representability over the rationals, so we get all numbers $\equiv -S\mod 8dd_1d_2d_3(a_3b_3c_3)^2$ can not be represented over the rationals. The weaker statement is that all number $\equiv -S\mod d_1d_2d_3S^2$ can not be represented over the rationals. To show the more particular results notice that $d_1d_2d_3S$ divides $abc$. For the last one, multiply these numbers by $R^2$, which is a square, so they remain unrepresentable over the rationals. We remark that if $a'b'c'$ is an even number, we may reduce the factor of $8$ in the modulus to $4$. 
		\end{proof}
		
		Finally, if $Q(x,y,z)$ is a general quadratic form over the rational, a linear change of variables with rational coefficients will transform it to a diagonal quadratic form $ax^2+by^2+cz^2$ with $a,b$ and $c$ rational numbers, see \cite{K}. These two forms have the same values over the rationals and using the relation $R(a,b,c)=dd_1d_2d_3R(a',b',c')$ in equation 1, we can find exactly the rational numbers represented by $Q$, if we use Theorem 8. If $Q$ is positive definite, then $a$, $b$ and $c$ are positive and Corollary 9 implies the following result.
		\begin{thm}
			Any positive definite ternary quadratic form over the rational numbers fails to represent an infinite progression of positive integers over the rational numbers.
		\end{thm}

	\end{document}